\documentclass[11pt]{amsart}
\usepackage{amssymb}
\usepackage{amssymb}

\begin{document}
\newcommand{\dbar}{\ensuremath{\overline\partial}}
\newcommand{\dbarstar}{\ensuremath{\overline\partial^*}}
\newcommand{\de}{\ensuremath{\partial}}
\newcommand{\C}{\ensuremath{\mathbb{C}}}
\newcommand{\R}{\ensuremath{\mathbb{R}}}
\newcommand{\D}{\ensuremath{\mathbb{D}}}
\newcommand{\T}{\ensuremath{\mathbb{T}}}

\makeatletter
\newcommand{\sumprime}{\if@display\sideset{}{'}\sum%
            \else\sum'\fi}
\makeatother

\numberwithin{equation}{section}

\newtheorem{theorem}{Theorem}[section]
\newtheorem{proposition}[theorem]{Proposition}
\newtheorem{conjecture}[theorem]{Conjecture}
\def\theconjecture{\unskip}
\newtheorem{corollary}[theorem]{Corollary}
\newtheorem{lemma}[theorem]{Lemma}
\newtheorem{observation}[theorem]{Observation}
\theoremstyle{definition}
\newtheorem{definition}{Definition}
\numberwithin{definition}{section}
\newtheorem{remark}{Remark}
\def\theremark{\unskip}
\newtheorem{question}{Question}
\def\thequestion{\unskip}
\newtheorem{example}{Example}
\def\theexample{\unskip}
\newtheorem{problem}{Problem}

\def\vvv{\ensuremath{\mid\!\mid\!\mid}}
\def\intprod{\mathbin{\lr54}}
\def\reals{{\mathbb R}}
\def\integers{{\mathbb Z}}
\def\N{{\mathbb N}}
\def\complex{{\mathbb C}\/}
\def\CP{{\mathbb{CP}}\/}
\def\P{{\mathbb P}\/}
\def\dist{\operatorname{dist}\,}
\def\spec{\operatorname{spec}\,}
\def\interior{\operatorname{int}\,}
\def\trace{\operatorname{tr}\,}
\def\df{\operatorname{DF}}
\def\cl{\operatorname{cl}\,}
\def\essspec{\operatorname{esspec}\,}
\def\range{\operatorname{\mathcal R}\,}
\def\kernel{\operatorname{\mathcal N}\,}
\def\dom{\operatorname{\mathcal D}\,}
\def\linearspan{\operatorname{span}\,}
\def\lip{\operatorname{Lip}\,}
\def\sgn{\operatorname{sgn}\,}
\def\Z{ {\mathbb Z} }
\def\e{\varepsilon}
\def\p{\partial}
\def\rp{{ ^{-1} }}
\def\Re{\operatorname{Re\,} }
\def\Im{\operatorname{Im\,} }
\def\dbarb{\bar\partial_b}
\def\eps{\varepsilon}
\def\Lip{\operatorname{Lip\,}}

\def\Hs{{\mathcal H}}
\def\E{{\mathcal E}}
\def\scriptu{{\mathcal U}}
\def\scriptr{{\mathcal R}}
\def\scripta{{\mathcal A}}
\def\scriptc{{\mathcal C}}
\def\scriptd{{\mathcal D}}
\def\scripti{{\mathcal I}}
\def\scriptk{{\mathcal K}}
\def\scripth{{\mathcal H}}
\def\scriptm{{\mathcal M}}
\def\scriptn{{\mathcal N}}
\def\scripte{{\mathcal E}}
\def\scriptt{{\mathcal T}}
\def\scriptr{{\mathcal R}}
\def\scripts{{\mathcal S}}
\def\scriptb{{\mathcal B}}
\def\scriptf{{\mathcal F}}
\def\scriptg{{\mathcal G}}
\def\scriptl{{\mathcal L}}
\def\scripto{{\mathfrak o}}
\def\scriptv{{\mathcal V}}
\def\frakg{{\mathfrak g}}
\def\frakG{{\mathfrak G}}

\def\ov{\overline}

\author{Siqi Fu and Mei-Chi Shaw}
\thanks
{The authors were supported in part by NSF grants.}

\address{Department of Mathematical Sciences,
Rutgers University-Camden, Camden, NJ 08102}
\email{sfu@camden.rutgers.edu}
\address{Department of Mathematics, University of Notre Dame, Notre Dame, IN 46556}
\email{Mei-Chi.Shaw.1@nd.edu}

\title[]  
{The Diederich-Forn{\ae}ss exponent and non-existence of Stein domains with Levi-flat boundaries}

\begin{abstract} We study the Diederich-Forn{\ae}ss exponent and relate it to non-existence of Stein domains with Levi-flat boundaries in complex manifolds. In particular, we prove that if the Diederich-Forn{\ae}ss exponent of a smooth bounded Stein domain in an $n$-dimensional complex manifold is $>k/n$, then it has a boundary point at which the Levi-form has rank $\ge k$.
\end{abstract}
\maketitle
\bigskip

%
%



\section{Introduction}\label{sec:intro}
A domain in a complex manifold is Stein if and only if there exists a smooth strictly plurisubharmonic exhaustion function. A Stein domain is called
hyperconvex if there exists a smooth bounded strictly plurisubharmonic function.
Diederich and Forn{\ae}ss \cite{DiederichFornaess77} showed that for any bounded pseudoconvex domain in $\C^n$ with $C^2$ boundary, there exist a positive constant $\eta$ and a  defining function $\rho$ such that $\hat\rho=-(-\rho)^\eta$ is plurisubharmonic on $\Omega$ (see also \cite{Range}).  The existence of bounded plurisubharmonic function   was later generalized to bounded pseudoconvex domains with $C^1$ boundary by Kerzman and Rosay \cite{KerzmanRosay81}  and with Lipschitz boundary by Demailly   \cite{Demailly87}  (see also  more recent results  by Harrington  \cite{Harrington07}).  The constant $\eta$ is called a  Diederich-Forn{\ae}ss exponent.
The supremum of all Diederich-Forn{\ae}ss exponents is called the Diederich-Forn{\ae}ss index of $\Omega$.
The Diederich-Forn{\ae}ss index has implications in regularity theory of the $\dbar$-Neumann Laplacian (see, for example, \cite{Kohn99, BerndtssonCharpentier00,  CSW04,
Harrington11}), as well as in estimates of the pluri-complex Green function \cite{Blocki04} and comparison of the Bergman and Szeg\"{o} kernels \cite{ChenFu11}.
The Diederich-Forn{\ae}ss indices can be arbitrarily small on the worm domains (\cite{DiederichFornaess77, DiederichFornaess77b}).   Sibony proved that for a smooth bounded pseudoconvex domain in $\C^n$ satisfying Property ($P$), the Diederich-Forn{\ae}ss index is one (see \cite{Sibony87}).

 If the  pseudoconvex  domain $\Omega$ has a  defining function which is bounded plurisubharmonic on $\overline\Omega$, then
the Diederich-Forn{\ae}ss index is one.   Forn{\ae}ss and Herbig \cite{FornaessHerbig08} showed that a smooth bounded domain in $\C^n$ with a defining function that is plurisubharmonic on the boundary also has Diederich-Forn{\ae}ss index one.
In this case, Boas and Straube showed that the  $\dbar$-Neumann Laplacian is global hypoelliptic on $L^2$-Sobolev spaces (see \cite{BoSt91}).
It was shown by Nemirovskii~\cite[Corollary]{Nemirovski99} that any smooth bounded Stein domain with a defining function that is plurisubharmonic on the domain cannot have Levi-flat boundary. In this paper, we study the Diederich-Forn{\ae}ss exponent and relate it to non-existence of Stein domains with Levi-flat boundaries in complex manifolds. Our main result can be stated as follows:

\begin{theorem}\label{th:main} Let $\Omega$ be a bounded Stein domain with $C^2$ boundary in a complex manifold $M$ of dimension $n$. If the Diederich-Forn{\ae}ss index of $\Omega$ is greater than $k/n$, $1\le k \le n-1$, then $\Omega$ has a boundary point at which the Levi form has rank   greater than $k$.
\end{theorem}

In particular, we have the following corollary.

\begin{corollary}\label{co:1} If the Diederich-For{\ae}ss index is greater than $1/n$, then its boundary cannot be Levi flat; and if the Diederich-For{\ae}ss exponent is greater than $1-1/n$, then its boundary must have at least one strongly pseudoconvex boundary point.
\end{corollary}

We would like to thank Professor Takeo Ohsawa who kindly informed us that similar results were obtained by Adachi and Brinkschulte independently using different methods~\cite{AB14}. For related work on the nonexistence of Levi-flat hypersurfaces in complex manifolds, we refer to the reader to papers \cite{Lins99, Nemirovski99, Siu00, Siu02, NiWo03, CSW04, CS07, Ohsawa07, Ohsawa13} in the references.

\section{The Diederich-Forn{\ae}ss index}\label{sec:df}

Let $M$ be an $n$-dimensional complex manifold with hermitian metric $\omega$. Let $\Omega$ be a bounded domain in $M$. A continuous real-valued function $r$ on $M$ is called a defining function of $\Omega$ if  $r<0$ on $\Omega$, $r>0$ on $M\setminus\overline \Omega$, and $C_1\delta(z)\le |r(z)|\le C_2 \delta(z)$ near $b\Omega$, where $\delta (z)$ is the geodesic distance from $z$ to the boundary $b\Omega$. We will also assume that the defining function $r$ is in the same smoothness class as that of the boundary $b\Omega$. A defining function $r$ is said to be normalized if $\lim_{z\to b\Omega} |r(z)|/\delta (z) =1$. Note that the signed distance function $\rho(z)=-\delta(z)$ on $\Omega$ and $\rho(z)=\delta(z)$ on $M\setminus\Omega$ is a normalized defining function for $\Omega$.

A constant $0<\eta\le 1$ is called a {\it Diederich-Forn{\ae}ss exponent} of a defining function $r$ of $\Omega$ if there exists a neighborhood $U$ of $b\Omega$ such that
\begin{equation}\label{eq:df}
\de\dbar(-(-r)^\eta)\ge 0
\end{equation}
on $U\cap\Omega$ in the sense of distribution.  We will call the supremum of all such $\eta$'s {\it the Diederich-Forn{\ae}ss index} of $r$ and denote it by $I(r)$.  The supremum of $I(r)$ over all defining functions of $\Omega$ is called {\it the Diederich-Forn{\ae}ss index} of $\Omega$ and is denoted by $I(\Omega)$.

A defining function $r$ is said to satisfy the {\it strong Oka property} if there exists a constant $K$ and a neighborhood $U$ of $b\Omega$ such that
\begin{equation}\label{eq:oka}
\p\dbar (-\log(-r))\ge K\omega
\end{equation}
on $U\cap\Omega$ in the sense of distribution.  The supremum of all such $K$'s is called the {\it Oka index} of $r$ and is denoted by $K(r)$. By Takeuchi's theorem, the signed distance function of a (proper) pseudoconvex domain in $\CP^n$ with the Fubini-Study metric satisfies the strong Oka property with  Oka index 1/12. (Hereafter, the Fubini-Study metric is normalized so that its holomorphic sectional curvature is $2$ and hence its holomorphic bisectional curvature is $\ge 1$.)

Let $\Omega\subset\subset M$ be a bounded domain with $C^2$-boundary. Let $r$ be a defining function of $\Omega$. Let $\omega_\nu=\de r/|\de r|$.
Let $L_\nu$ be the dual vector of $\omega_\nu$.  For any $(1, 0)$-vector $X$ near $b\Omega$,  let
$X_\nu=\langle X, L_\nu\rangle_\omega L_\nu$ be the complex normal component of $X$ and $X_\tau=X-X_\nu$ the
complex tangential component.  Write $T^{1, 0}(r)=\{(z, X)\in T^{1, 0}(M)\mid Xr=0\}$.  For $z\in b\Omega$, we further decompose $X_\tau=X_s+X_l$,  where $X_l$ is in the the null space $\scriptn_z$ of the Levi-form $\de\dbar\delta$ at $z$ and $X_s\perp X_l$.  Let $S^{1, 0}(M)=\{(z, X)\in T^{1, 0}(M), |X|_\omega=1\}$. Let $W$ be the weakly pseudoconvex points on $b\Omega$. Let
$$
S(r)=\max\{|\de\dbar r(X_l, \overline{L}_\nu)(z)|;  \quad |X_l|_\omega=1, X_l\in\scriptn_z, z\in W\}.
$$
If $b\Omega$ is strongly pseudoconvex, we set $S(r)=0$.  Define
\begin{equation}\label{eq:oo}
I_0(r)=\max\Big\{\min\big\{\frac{K(r)}{8(S(r))^2}, \frac{1}2\big\}, \ 1-\frac{2(S(r))^2}{K(r)}\Big\}>0.
\end{equation}

\begin{theorem}\label{prop:oka}  Let $\Omega$ be a bounded domain with $C^2$-boundary and let $r$ be a normalized defining function that satisfies the strong Oka property.  Then  $I(r)\ge I_0(r)$.
\end{theorem}

\begin{proof} A simple computation yields that
\begin{equation}\label{eq:c2}
\p\dbar(-\log(-r))=\frac{\p\dbar r}{-r}+
\frac{\p r\wedge\dbar r}{r^2}
\end{equation}
and
\begin{align}\label{eq:c4}
\p\dbar(-(-r)^\eta)&=\eta(-r)^{\eta}\Big(\frac{\p\dbar r}{-r}+(1-\eta)
\frac{\p r\wedge\dbar r}{r^2}\Big)\notag\\
&=\eta(-r)^\eta\Big(\p\dbar(-\log(-r))-
\eta\frac{\p r\wedge\dbar r}{r^2}\Big).
\end{align}
It follows from \eqref{eq:c4} that \eqref{eq:df} is equivalent to
\begin{equation}\label{eq:o1}
\p\dbar (-\log (-r))\ge \eta \frac{\de r\wedge\dbar r}{r^2}.
\end{equation}

Let $c_0$ be a constant such that $0<c_0<K(r)$.  Then
\begin{equation}\label{eq:oka-alt}
\p\dbar (-\log(-r))\ge c_0\omega
\end{equation}
for $z\in\Omega$ near the boundary.  It follows from \eqref{eq:c2} that
\begin{equation}\label{eq:o2}
\frac{\de\dbar r(X_\tau, \overline{X}_\tau)}{-r}\ge c_0|X_\tau|^2_\omega.
\end{equation}
Let $C_1$ be any constant such that $C_1> S(r)$. Then there exists a neighborhood $U$ of $\scriptn^{1, 0}(W)=\{(z, X)\mid  z\in W, X\in \scriptn_z, |X|_\omega=1\}$ in $S^{1, 0}(M)$ such that
\begin{equation}\label{eq:o3}
|\de\dbar r(X, \ov{L}_\nu)|\le C_1, \qquad (z, X)\in U.
\end{equation}
For $(z, X_\tau)\in S^{1, 0}(\ov{\Omega})\setminus U$ with $z$ near $b\Omega$,
\begin{equation}\label{eq:o4}
\de\dbar r(X_\tau, \ov{X}_\tau)\ge C_2 |X_\tau|^2_\omega
\end{equation}
for some constant $C_2>0$.  We write $X=X_\tau+X_\nu$ with $X_l\in\scriptn_z$
as before.  Then
\begin{align}
\de\dbar (-\log(-r))(X, \ov{X})&=\frac{\de\dbar r(X_\tau, \ov{X}_\tau)}{-r}+\frac{\de\dbar r(X_\nu, \ov{X}_\nu)}{-r}\notag\\
&+\frac{2\Re\de\dbar r(X_\tau, \ov{X}_\nu)}{-r}+\frac{|Xr|^2}{r^2}.\label{eq:o5}
\end{align}
Note that $|Xr|=|X_\nu|_\omega \cdot |\partial r|_\omega$. Let $K_0=\sup\{|\partial\dbar r|_\omega;\ z\in \overline{\Omega}\}$. Then
\begin{equation}\label{eq:o6}
|\de\dbar r(X_\nu, X_\nu)|\le K_0 |Xr|^2/|\partial r|^2_\omega
\end{equation}
Similarly,
\begin{equation}\label{eq:o6a}
|\Re\de\dbar r(X_\tau, \ov{X}_\nu)| \le K_0 |X_\tau|_\omega\cdot |Xr|/|\de r|_\omega.
\end{equation}

We first deal with the strictly pseudoconvex directions. For $(z, X)\in T^{1, 0}(\Omega)$
with $ (z, X_\tau/|X_\tau|)\in S^{1, 0}(\Omega)\setminus U$ with $z$ near $b\Omega$,
it follows from \eqref{eq:o6a} and \eqref{eq:o4} that for any positive constant $M$,
\begin{equation}
\begin{aligned}\label{eq:o7}
|2\Re\de\dbar r(X_\tau, \ov{X}_\nu)| &\le K_0\left(\frac{1}{M}|X_\tau|^2_\omega+\frac{M}{|\de r|^2_\omega}|Xr|^2\right)\\
&\le \frac{K_0}{MC_2} \p\dbar r(X_\tau, \ov{X}_\tau)+\frac{K_0M}{|\de r|^2_\omega} |Xr|^2.
\end{aligned}
\end{equation}
Therefore,
\begin{equation}
\begin{aligned}\label{eq:o8}
\de\dbar (-\log (-r))(X, \ov{X})&\ge \left(1-\frac{K_2}{MC_2}\right)\frac{\de\dbar r(X_\tau, \ov{X}_\tau)}{-r}\\
&\qquad\qquad +\left(1-\frac{K_0(M+1) |r|}{|\de r|^2_\omega}\right) \frac{|Xr|^2}{r^2}.
\end{aligned}
\end{equation}
By choosing $M$ sufficiently large and then letting $z$ be sufficiently close to $b\Omega$, we know that \eqref{eq:o1} holds for any $\eta<1$.

We now deal with weakly pseudoconvex directions. For $(z, X)\in T^{1, 0}(\Omega)$ with $(z, X_\tau/|X_\tau|_\omega)\in U$, we have
\begin{equation}\label{eq:o8b}
2|\de\dbar r(X_\tau, \ov{X}_\nu)|\le 2C_1|X_\tau|_\omega|Xr|/|\partial r|_\omega\le C_1(\frac{|r|}{\eps}|X_\tau|_\omega^2+\frac{\eps}{|r|} \frac{|Xr|^2}{|\partial r|^2_\omega}),
\end{equation}
where $\eps$ is a positive constant to be chosen.  Since $r$ is a normalized defining function, $|\partial r|_\omega=1/\sqrt{2}$ on $b\Omega$. Combining \eqref{eq:o8b} with \eqref{eq:o2}, we have
\begin{align}\label{eq:o9}
\de\dbar(-\log(-r))(X, \ov{X})&\ge (c_0-C_1/\eps)|X_\tau|^2+ \frac{1-(C_1\eps+K|r|)|\de r|_\omega^{-2}}{r^2}|Xr|^2\notag\\
&\ge (c_0-C_1/\eps)|X_\tau|^2+ \frac{1-2C_1\eps-K'|r|}{r^2}|Xr|^2
\end{align}
for some positive constant $K'$.

We consider two cases: $4C_1^2\le c_0$ and $4C_1^2>c_0$.  When $4C_1^2\le c_0$, we take $\eps=C_1/c_0$. Then
\begin{equation}\label{eq:o9a}
\de\dbar(-\log(-r))(X, \ov{X})\ge (1-2C_1^2/c_0-K'|r|)|Xr|^2/r^2.
\end{equation}
When $4C_1^2>c_0$, we take $\eps=\frac{1}{4C_1}<C_1/c_0$. Then combining \eqref{eq:o9} with \eqref{eq:oka-alt}, we
have
\[
\de\dbar(-\log(-r))(X, \ov{X})\ge -\left(\frac{C_1}{c_0\eps}-1\right) \p\dbar (-\log (-r))(X, \ov{X})+\frac{1-2C_1\eps-K'|r|}{r^2}|Xr|^2.
\]
Therefore,
\[
\de\dbar(-\log(-r))(X, \ov{X})\ge \left(\frac{c_0\eps(1-2C_1\eps)}{C_1} -\frac{K'c_0\eps|r|}{C_1}\right)\frac{|Xr|^2}{r^2}.
\]
Hence
\begin{equation}\label{eq:o9b}
\de\dbar(-\log(-r))(X, \ov{X})\ge (\frac{c_0}{8C_1^2}-\frac{K'c_0\eps |r|}{2C_1^2})\frac{|X_\nu|^2}{r^2}.
\end{equation}
Note that when $4C_1^2\le c_0$, we have
\begin{equation}\label{eq:09c}
1- \frac{C_1^2}{c_0}\ge\frac{1}2 \ \text{ and }\ \frac{c_0}{4C_1^2}\ge \frac{1}2.
\end{equation}
Furthermore, when $4C_1^2>c_0$,
\begin{equation}\label{eq:o9d}
\frac{1}{2}>\frac{c_0}{8C_1^2}>1-\frac{C_1^2}{c_0}.
\end{equation}
Combing \eqref{eq:o9a}-\eqref{eq:o9d}, we know that \eqref{eq:o1} holds for any $\eta<I_0(r)$.  We thus conclude the proof of Proposition~\ref{prop:oka}
\end{proof}

By Takeuchi's theorem (\cite{Takeuchi64}, see also \cite{CaoShaw05, GreeneWu78}), \eqref{eq:oka} holds for the signed distance function with $K=1/12$ on any proper pseudoconvex domain on complex projective space $\CP^n$. Combing this with Proposition~\ref{prop:oka}, we have:

\begin{corollary} Let $\Omega$ be a proper pseudoconvex domain in $\CP^n$ with $C^2$ boundary. Then its Diederich-Fornaess index
\[
I(\Omega)\ge  I_0(\rho)=\max\Big\{\min\big\{\frac{1}{96(S(\rho))^2}, \frac{1}2\big\}, \ 1-24(S(\rho))^2\Big\}>0,
\]
where $\rho$ is the signed distance function to $b\Omega$ with respect to the Fubini-Study
metric.
\end{corollary}

\begin{proposition}\label{prop:oka2} Let $\Omega\subset\subset M$ be a bounded domain with $C^2$ boundary and let $r$ be a normalized defining function. Suppose \eqref{eq:oka} holds and there exist a neighborhood $V$ of the set $W$ of weakly pseudoconvex boundary points and a positive constant $K_1>1$
such that
\begin{equation}\label{eq:key}
K|X_\tau|^2\le \frac{\p\dbar r(X_\tau, \ov{X}_\tau)}{r}\le K K_1 |X_\tau|^2,
\end{equation}
for all $z\in V$ and $X\in T_z^{1, 0}(M)$.  Then
\begin{equation}\label{eq:key1}
I(\Omega)\ge \max\left\{\min\left\{\frac{1}{8(K_1-1)}, \frac{1}{2}\right\}, \ 3-2K_1\right\}.
\end{equation}
\end{proposition}

\begin{proof} From \eqref{eq:oka}, we know that
\[
\Theta =\de\dbar (-\log(-r)) - K\omega
\]
is positive semi-definite. Applying the Cauchy-Schwarz inequality to $\Theta(X_\tau, X_\nu)$, we then have
\[
|\Theta(X_\tau, X_\nu)|\le |\Theta(X_\tau, X_\tau)|^{1/2} |\Theta(X_\nu, X_\nu)|^{1/2}.
\]
(We refer the reader to \cite{Straube01} for a similar technique that has been used by Straube to construct Stein neighborhood bases in connection with regularity theory in the $\bar\partial$-Neumann problem.) Therefore,
\[
\left|\frac{\de\dbar r(X_\tau, \ov{X}_\nu)}{r}\right|^2 \le \left(\frac{\de\dbar r(X_\tau, X_\tau)}{-r}-K|X_\tau|^2_\omega\right)\left(\frac{\de\dbar r(X_\nu, X_\nu)}{-r}+\frac{|X r|^2}{r^2}-K|X_\nu|^2_\omega\right).
\]
Thus
\[
\left|\de\dbar r(X_\tau, \ov{X}_\nu)\right|\le ((K_1-1)K)^{1/2} (1+C|r|)^{1/2} |X_\tau||X_\nu|,
\]
for some positive constant $C$.  The inequality \eqref{eq:key1} then follows by applying Proposition~\ref{prop:oka} with $S(r)=((K_1-1)K)^{1/2}$.
\end{proof}

Let $f\in C^2(M)$. Recall that the real Hessian $H_f$ is defined by
\[
H_f(\xi, \zeta)(z)=\langle \nabla_\xi(\nabla f), \zeta\rangle
\]
for $\xi, \zeta\in T_\R(M^{2n})$, where $\nabla_\xi$ denotes the covariant derivative.  For any $X\in T^{1, 0}_{\C}(M)$,
we write $X=\frac{1}{\sqrt{2}}(\xi_X-\sqrt{-1}J\xi_X)$ where $J$ is
the complex structure.  Let $z$ be a point in $\Omega$ near the boundary
and $\pi(z)$ be its closest point on $b\Omega$. Let $\gamma(t)$
be the geodesic parametrized by arc-length such that $\gamma(0)=\pi(z)$. For any $(1, 0)$ tangent vector $X$ at $z$
near $b\Omega$, we let $X(t)$ be the vector at $\gamma(t)$ obtained by parallel translate (of real and imaginary parts) of $X$ along the geodesic from $z$ to $\gamma(t)$ and let
$X^0=X(0)$.

\begin{proposition}  Let $\Omega$ be a proper
pseudoconvex domain with $C^2$ boundary in $\CP^n$. Let $\rho$ be the signed distance
function to $b\Omega$ with respect to the Fubini-Study metric. Let
\[
K_2=\max\{|\nabla_{\xi_X} (\nabla\rho)|^2_\omega+|\nabla_{J\xi_X}(\nabla\rho)|^2_\omega; \ z\in W, X\in\scriptn_z, |X|_\omega=1\}.
\]
Then
\[
I(\Omega)\ge \max\left\{\min\left\{\frac{1}{8(K_2-1)}, \frac{1}{2}\right\}, \ 3-2K_2\right\}.
\]
\end{proposition}

\begin{proof} It follows from the computations in \cite{Weinstock75} that
\[
\lim_{t\to 0^+} \frac{1}{t}\left(\p\dbar \rho(X_\tau(t), X_\tau(t))-\p\dbar\rho(X^0, \ov{X}^0)\right)=|\nabla_{\xi_X} (\nabla\rho)|^2_\omega+|\nabla_{J\xi_X}(\nabla\rho)|^2_\omega
\]
(The above identity was proved in \cite{Weinstock75} for $\Omega$ in $\C^n$; compare also \cite{Straube01}. The proof for $\Omega$ in $\CP^n$ is similar; see \cite{CaoShaw05} for related arguments.) We then conclude the proof by applying Proposition~\ref{prop:oka2} with $K=1$ and any $K_1>K_2$. \end{proof}

From Proposition~\ref{prop:oka}, we also obtain the following slight variation of a result of Ohsawa and Sibony (\cite{OhsawaSibony98}; see also \cite{CSW04, CS07}):

\begin{corollary}\label{cor:oka1} Let $\Omega$ be a bounded domain in $M$ with $C^2$ boundary. Suppose $r$ is a normalized defining function that satisfies \eqref{eq:oka}. Then
for any $c\in (0,\ K)$ and $\eta\in (0,\ I_0(r))$, there exists a neighborhood $V$ of $b\Omega$ such that
\[
\de\dbar (-\log(-r))\ge c\omega+(1-\frac{c}{K})\eta\frac{\de r\wedge\dbar r}{r^2}
\]
and
\[
\de\dbar(-(-r)^\eta)\ge \eta(-r)^\eta\Big(c\omega+(1-\frac{c}{K})\eta\frac{\de r\wedge\dbar r}{r^2}\Big).
\]
\end{corollary}

\section{Non-existence of Stein domains with Levi-flat boundaries}

We prove Theorem~\ref{th:main} in this section. We first recall the following well-known simple lemma.  Let $\Omega$ be a bounded domain with $C^2$ boundary in a complex hermitian manifold $M$ of dimension $n$.  Let $\rho$ be a defining function for $\Omega$.  For $t>0$, let $\Omega_{-t}=\{z\in\Omega; \ \rho<-t\}$.  Let $i_t\colon b\Omega_{-t}\to M$ be the inclusion map. Let $1\le k\le n$ be an integer.

\begin{lemma}\label{lm:levi-rank} If the rank of the Levi form of $b\Omega$ is $\le k-1$ at all  $z\in b\Omega$, then
\begin{equation}\label{eq:levi-rank}
i_{t}^*(d^c\rho\wedge (d d^c\rho)^{n-1})=O(t^{n-k}) dS_t
\end{equation}
where $dS_t$ is the surface element of $b\Omega_{-t}$.
\end{lemma}

We sketch the proof for the reader's convenience. Note that $dS_t=i_{t}^*(* dr)/|dr|_\omega$
and
\[
i_{t}^*(d^c\rho\wedge (d d^c\rho)^{n-1})=\nu\lrcorner ((d\rho/|d\rho|)\wedge d^c\rho\wedge (dd^c\rho)^{n-1})
\]
where $\nu$ is the dual vector of $d\rho/|d\rho|_\omega$. By choosing local holomorphic coordinates that diagonalize the Levi form, we then obtain
\eqref{eq:levi-rank}.

We now prove Theorem~\ref{th:main}. Let $\rho$ be a defining function of $\Omega$ such that $\hat\rho=-(-\rho)^\eta$ is plurisubharmonic on $\Omega$ for some constant $\eta>k/n$. Let $\Omega_{-t}=\{\rho<-t\}$, $t>0$. Since $\Omega$ is Stein, $\Omega_{-t}$ has at least a strictly pseudoconvex boundary point for sufficiently small $t$. Let
\[
f(t)=\int_{\Omega_{-t}} (d d^c\hat\rho)^n.
\]
Then $f(t)\ge 0$ and $f(t)$ is decreasing. By Stokes' theorem,
\[
f(t)=\int_{b\Omega_{-t}} i^*_t( d^c\hat\rho\wedge (dd^c\hat\rho)^{n-1}).
\]
Since
\[
d^c\hat\rho=i\eta (-\rho)^{\eta-1} (\dbar\rho-\p\rho) \ \ \text{ and } \ \ dd^c\hat\rho=2i\eta\rho^\eta\left(\frac{\p\dbar\rho}{-\rho}+
(1-\eta)\frac{\p\rho\wedge\dbar\rho}{\rho^2}\right),
\]
we have
\[
d^c\hat\rho\wedge\left(dd^c\hat\rho\right)^{n-1}=\eta^n (-\rho)^{n(\eta-1)} d^c\rho\wedge\left(dd^c\rho\right)^{n-1}.
\]
Suppose the Levi rank of $b\Omega$ is $\le k-1$ at all boundary points, then by Lemma~\ref{lm:levi-rank},
\[
i^*_t(d^c\rho\wedge\left(dd^c\rho\right)^{n-1})=O(t^{n-k}) dS_t.
\]
Thus
\[
f(t)=O(t^{n\eta-k}).
\]
Therefore, $\lim_{t\to 0^+} f(t)=0$ and hence $f(t)=0$ for small $t>0$. This implies that $b\Omega_{-t}$ has Levi rank $\le n-2$ at each point, which leads to a contradiction. This concludes the proof of Theorem~\ref{th:main}.

Corollary 1.2 follows easily. The following theorem  is a  variation of Theorem~\ref{th:main}.

\begin{theorem}\label{theorem:2} Let $M$ be a complex manifold of dimension $n$ with a hermitian metric $\omega$.  Let $\Omega$ be a bounded Stein domain in $M$ with $C^2$ boundary. Suppose there exist a defining function $\rho$, a constant $\eta>0$, and a neighborhood $U$ of $b\Omega$ such that
\begin{equation}\label{eq:oka1}
\de\dbar(-(-\rho)^\eta)\ge c(-\rho)^\eta\Big(\omega+\frac{\de\rho\wedge\dbar\rho}{\rho^2}\Big)
\end{equation}
on $U\cap\Omega$ for some constant $c>0$.  If $\eta\ge 1/n$,
then $\Omega$ cannot have Levi-flat boundary.
\end{theorem}
\begin{proof}
In light of Theorem~\ref{th:main}, it remains to prove the case when $\eta=1/n$. We follow the notations as in the above proof of Theorem~\ref{th:main}.
 Let $\eps_0$ be sufficiently small such that $\Omega\setminus\Omega_{-\eps_0}\subset U\cap\Omega$. We set
\[
 f(t)=\int_{\Omega_{-t}\setminus\Omega_{-\eps_0}} (dd^c\widehat\rho)^{n}
\]
for $0<t<\eps_0$. Suppose $b\Omega$ is Levi-flat, then as in the proof of Theorem~\ref{th:main},
\[
\left. d^c\hat\rho\wedge\left(dd^c\hat\rho\right)^{n-1}\right|_{b\Omega_{-t}}=\left.\eta^n (-\rho)^{n(\eta-1)} d^c\rho\wedge\left(dd^c\rho\right)^{n-1}\right|_{b\Omega_{-t}}=O(t^{n\eta-1})\,dS_t\le C\, dS_t .
\]
By Stoke's theorem
\begin{equation}\label{eq:1}
f(t)=\int_{b\Omega_{-t}} d^c\hat\rho\wedge (dd^c\hat\rho)^{n-1}-\int_{b\Omega_{-\eps_0}}d^c\hat\rho\wedge (dd^c\hat\rho)^{n-1} \le C.
\end{equation}

On the other hand, it follows from \eqref{eq:oka1} that
\[
(d d^c\widehat\rho)^n\ge C\delta^{n\eta}\Big(\omega+\frac{\de\delta\wedge\dbar\delta}{\delta^2}\Big)^n\ge C\delta^{n\eta-2}dV,
\]
where  $dV$ is the volume element. Thus
\[
\aligned
f(t)&=\int_{\Omega_{-t}\setminus\Omega_{-\eps_0}} (d d^c\widehat\rho)^n\ge C\int_{\Omega_{-t}\setminus\Omega_{-\eps_0}} (-\rho)^{n\eta-2}dV
\\&\ge C\int_{-\eps_0}^{-t}(-\rho)^{-1}d\rho  \ge C(-\log t+\log \eps_0).
\endaligned
\]
Therefore, $\lim_{t\to 0^+} f(t)=\infty$, which leads to a  contradiction with \eqref{eq:1}. This concludes the proof of Proposition~\ref{theorem:2}.

\end{proof}
\bigskip
\noindent{\bf Acknowledgement:}  This work was done while the first author visited the University of Notre Dame in April, 2012. He thanks the Department of Mathematics for the warm hospitality.

\bibliography{survey}
\providecommand{\bysame}{\leavevmode\hbox
to3em{\hrulefill}\thinspace}

\end{document}